\documentclass{nmj_my}

\usepackage{hyperref}
\usepackage{amsmath,amssymb,amscd}
\usepackage[all]{xy}

\usepackage{graphicx} 
\usepackage{color}
\usepackage{version}
\usepackage{url}

\excludeversion{NB}

\theoremstyle{definition}
\newtheorem{thm}{Theorem}[section]

\newtheorem{ex}[thm]{Example}
\newtheorem{prop}[thm]{Proposition}

\newtheorem{rem}[thm]{Remark}

\renewcommand{\proofname}{\textit{Proof.}}

\newcommand{\mr}[1]{{\mathrm{#1}}}

\title[Hyperbolic $3$-manifolds and Cluster Algebras]{Hyperbolic $3$-manifolds and Cluster Algebras}

\author{
  \name[Kentaro Nagao]{Kentaro Nagao}
  \address{Graduate School of Mathematics, Nagoya University}
  \email{kentaron@math.nagoya-u.ac.jp}
} 

\author{
  \name[Yuji Terashima]{Yuji Terashima}
  \address{Department of Mathematical and Computing Sciences, Tokyo Institute of Technolog}
  \email{tera@is.titech.ac.jp}
} 

\author{
  \name[Masahito Yamazaki]{Masahito Yamazaki}
  \address{Kavli Institute for the Physics and Mathematics of the Universe, University of Tokyo}
  \email{masahito.yamazaki@ipmu.jp}
} 


\begin{document}

\maketitle


\begin{abstract}
We advocate the use of cluster algebras and their $y$-variables in the
 study of hyperbolic 3-manifolds.
We study hyperbolic structures on the mapping tori of pseudo-Anosov mapping classes of punctured surfaces,
and show that cluster $y$-variables naturally give the solutions of the
 edge-gluing conditions of ideal tetrahedra. We also comment on the
 completeness of hyperbolic structures.
\end{abstract}

\bigskip


\setcounter{section}{-1}

\section{Introduction}
\subsection{Cluster Algebras}
Cluster algebras were introduced by Fomin and Zelevinsky (\cite{fomin-zelevinsky1}) around 2000. 
Since then, many authors have uncovered beautiful connections between 
the theory of cluster algebras and a wide range of mathematics such as

\begin{itemize}
\item dual canonical bases and their relations with preprojective algebras and quiver varieties (\cite{fomin-zelevinsky3}, \cite{leclerc_ICM}, \cite{nakajima-cluster}, \cite{kimura})
\item total positivity (\cite{fomin_ICM})
\item (higher) Teichm\"{u}ller theory and its quantization (\cite{fock-goncharov-teichmuller,fock-goncharov-dual,fock-goncharov-quantum-cluster}, \cite{teschner,teschner2})
\item $2$-dimensional hyperbolic geometry (\cite{cluster-poisson}, \cite{FST})
\item cluster categories (\cite{keller-survay}, \cite{amiot}, \cite{plamondon})
\item discrete integrable systems (\cite{kedem}, \cite{kns})
\item Donaldson-Thomas theory (\cite{ks}, \cite{cluster-via-DT})
\item supersymmetric gauge theories (\cite{GMN2}, \cite{CNV},
      \cite{Eager})
\end{itemize}
The goal of this paper is to add yet another item to this list: \emph{the
theory of hyperbolic $3$-manifolds}. This paper is a companion to
\cite{NTY}, which discusses the application of cluster algebras to the
physics of 3d $\mathcal{N}=2$ supersymmetric gauge theories.

\subsection{\texorpdfstring{Hyperbolic $3$-manifolds}{Hyperbolic 3-manifolds}}
A hyperbolic $3$-manifold (with cusps) has a decomposition into ideal
tetrahedra. This makes it possible for us to compute invariants of the $3$-manifold, 
such as the hyperbolic volume and the Chern-Simons invariant
\cite{NZ_volume,neumann_CS}. 

An ideal tetrahedron is parametrized by a complex number 
called a {\it shape parameter}. 
Given a topological decomposition of the $3$-manifold into ideal tetrahedra, 
we need to find shape parameters which satisfy {\it edge-gluing equations} (\S \ref{subsec_hyperbolic}) 
in order to obtain a hyperbolic structure on the $3$-manifold.
Moreover, 
the {\it cusp equations} (\S \ref{subsec_complete}) should hold for the complete hyperbolic structure.
In general, it is a rather non-trivial problem to systematically 
find solutions of these equations.

In this paper, we study mapping tori $M_\varphi$ of mapping classes $\varphi$ of
a surface $\Sigma$ with punctures. We mainly discuss the case that the mapping torus admits a hyperbolic structure.

The main results of this paper are summarized as follows: 
\begin{itemize}
\item 
Solving the periodicity equation in Theorem \ref{thm_main} for cluster 
transformations, we get a solution of the edge-gluing equations of 
the mapping torus $M_\varphi$ with an ideal 
triangulation induced by the cluster transformations. 
\item Shape parameters of tetrahedra are given by the cluster
      $y$-variables, where the initial values of the $y$-variables are
      taken to be the solution of the periodicity equation.
\item The cusp condition is written as a simple condition on a product
      of the initial
      values of the $y$-variables.
\end{itemize}

\begin{rem}
\label{rem.KT}
The complete hyperbolic structure gives a non-zero solution 
of the periodicity equation, thanks to the result of \cite[Cor 2.6]{kitayama-terashima}.
\end{rem}

\begin{rem}
This paper has grown up from our attempts to formulate
the results of \cite{nakanishi_kashaev} and \cite{TY1,TY2} in mathematically rigorously.

In \cite{TY1}, the authors conjectured 
an equivalence of the partition function of a $3d$ $\mathcal{N} = 2$ gauge theory on a duality wall and that of the $\mathrm{SL}(2,\mathbb{R})$ Chern-Simons theory on a mapping torus.
This is a $3d/3d$ counterpart of the $4d/2d$ correspondence, known as the AGT relation (\cite{AGT}).

 In \cite{TY2}, the authors demonstrated that a limit of the 3d N = 2 partition
function reproduces the hyperbolic volume of the mapping torus in the case
of the once-punctured torus by using quantum cluster transformations.
The key observation in \cite{TY2} was that the shape parameters satisfying
edge-gluing equations (as previously analyzed in \cite{one_punctured}) appear at the
saddle point.

In \cite{nakanishi_kashaev}, it was shown that classical
dilogarithm identities (\cite{nakanishi-periodicity}) naturally emerge
from quantum dilogarithm identities (\cite{keller-q-dilog}, \cite{dilog_id}) 
by the saddle-point method. 

It will be interesting to learn from physics about the ``quantum'' aspects of hyperbolic geometry of $3$-manifolds.

\end{rem}

\begin{rem}\label{rem.TBA}
There is a known relation between cluster transformations and 
integrable systems \cite{FZ_Ysystem,Keller_periodicity,Nakanishi_id,kedem}.
With this, our theorem, which connects 
cluster transformations to $3$-manifolds, 
gives a natural explanation for mysterious and interesting relations 
between $3$-manifolds and 
conformal field theories/integrable systems, originally found in \cite{GT,NahmRT,Dupont}. 
We illustrate this point by an example in the final subsection (the corresponding $3$-manifold is not hyperbolic).

The differences between the two setups, (a) integrable systems and (b) hyperbolic $3$-manifolds,  can be stated in several different languages:
\begin{itemize}
\item
We have periodicity conditions on the cluster $y$-variables both in (a) and in (b).
However, in (a), periodicity is 
imposed as identities of rational functions on $y_i$'s, 
whereas in (b) we solve the periodicity equations to determine values of
     $y_i$, which in turn determines the hyperbolic structure of the mapping tori.

\item
In (a), the product of the quantum dilogarithms associated to the sequence of mutations is equal to $1$ 
(quantum dilogarithm identity \cite{keller-q-dilog}). 
In (b), the product gives a non-trivial action of the mapping class in the quantum
     Teichm\"{u}ller theory.
     
\item 
In terms of surface triangulations and flips, after a sequence of flips,
in (a) we get the original triangulation (up to a permutation of
      vertices), while in (b) we get the original triangulation 
      pulled back by the mapping class.

\item 
A mutation provides a derived equivalence of $3$-dimensional Calabi-Yau categories associated to quivers with potential (\cite{dong-keller}).
In (a), the composition of the derived equivalences is an identity functor, while in (b) it gives the action of the mapping class on the derived category (see \cite{MCG_DT}).

\item 
The derived equivalence induced by mutation corresponds to a
      wall-crossing in the space of stability conditions, 
and a sequence of mutations gives a new chamber.
In (a), the new chamber coincides with the original one, while in (b) 
the chamber is obtained from the original one by the action of the mapping class on the space of stability
      conditions.
In other words, the former is the wall crossing associated with a contractible cycle, whereas the latter corresponds to a non-contractible cycle with non-trivial monodromies ({\it cf.} \cite{galaxy}). \end{itemize}
\end{rem}

\section*{Acknowledgments}

K.\ N.\ is supported by the Grant-in-Aid for Research Activity Start-up (No.\ 22840023) and for Scientific Research (S) (No.\ 22224001). 
Y.\ T.\ is supported in part by the Grants-in-Aid for Scientific
Research, JSPS (No 22740036).
M.\ Y.\ would like to thank PCTS and its anonymous donor for generous support.
We would like to thank H.\ Fuji, A.\ Kato, S.\ Kojima, H.\
Masai, T.\ Nakanishi, S.\ Terashima, T.\ Yoshida and D.\ Xie for helpful conversation.

\section{Cluster Algebras}
\subsection{Quiver Mutation}\label{subsec_qm}
In this paper, we always assume that a quiver has
\begin{itemize}
\item the vertex set $I=\{1,\ldots,n\}$, and
\item no loops and oriented $2$-cycles (see Figure \ref{fig_loops}).
\end{itemize}
\begin{figure}[htbp]
  \centering
  \input{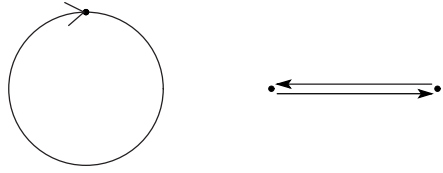}
  \caption{A loop (left) and an oriented $2$-cycle (right) of a quiver.}
\label{fig_loops}
\end{figure}
For vertices $i$ and $j\in I$, we define
\[
Q(i,j)=\sharp\{\text{arrows from $i$ to $j$}\},\quad \overline{Q}(i,j)=Q(i,j)-Q(j,i).
\]
Note that the quiver $Q$ is uniquely determined by the skew-symmetric matrix
$\overline{Q}(i,j)$ (or equivalently $Q(i,j)$) under the assumption above.\footnote{In this paper
we restrict ourselves to cluster algebras associated with skew-symmetric matrices.}

For the vertex $k$, we define a new quiver $\mu_kQ$ (mutation of $Q$
at vertex $k$) by an
anti-symmetric matrix
\[
\overline{\mu_kQ}(i,j)=
\begin{cases}
-\overline{Q}(i,j) &  \text{$i=k$ or $j=k$}, \\
\overline{Q}(i,j) + Q(i,k)Q(k,j) - Q(j,k)Q(k,i) &  \text{$i,j\neq k$}.
\end{cases}
\]

\subsection{Cluster Variables}\label{subsec_cv}
Given a sequence $\mathbf{k}=(k_1,\ldots,k_l)$ of vertices
and ``time'' parameters $t=0,\ldots, l$, we define
\[
Q_0:=Q, \quad Q_t:=\mu_{k_{t-1}}\cdots\mu_{k_1}Q \quad (t>0) .
\] 
For initial values $x_i(0)=x_i$ and $y_i(0)=y_i$, 
we define the cluster $x$-variables $x_i(t)$ and the cluster
$y$-variables (coefficients) $y_i(t)$ ($i\in I$) by
\begin{equation}\label{eq_x}
x_i(t+1)=\displaystyle\frac{\prod_jx_j(t)^{Q_t(i,j)}+\prod_j x_j(t)^{Q_t(j,i)}}{x_i(t)},
\end{equation}
and
\begin{equation}\label{eq_y}
y_i(t+1)=
\begin{cases}
y_k(t)^{-1} & i=k,\\
y_i(t) \, y_k(t)^{{Q_t}(k,i)} \, \bigl(1+y_k(t)\bigr)^{\overline{Q_t}(i,k)} & i\neq k.
\end{cases}
\end{equation}


\section{Triangulated Surfaces and Quivers}
Let $\Sigma$ be a closed connected oriented surface and $M$ be a finite set of points on $\Sigma$, called {\it punctures}. 
We assume that $M$ is non-empty and $(\Sigma,M)$ is not a sphere with less than four punctures.

We choose an ideal triangulation $\tau$ of $\Sigma$, 
{\it i.e.}, we decompose $\Sigma$ into triangles whose vertices
are located at the punctures.
We will not allow self-folded arcs (see Figure \ref{fig_self}) in this paper.
\begin{figure}[htbp]
  \centering
  \input{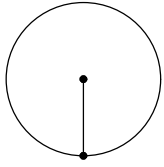}
  \caption{We do not allow self-folded triangles as in this Figure.}
\label{fig_self}
\end{figure}

\subsection{Quiver Associated to a Triangulation}
For a triangulation $\tau$ without self-folded arcs
we will define a quiver $Q_\tau$ whose vertex set $I$ is the set of arcs in $\tau$.

For a triangle $\Delta$ and arcs $i$ and $j$, we define a skew-symmetric integer matrix $\overline{Q}^\Delta$ by 
\[
\overline{Q}^\Delta(i,j):=
\begin{cases}
1 & \text{$\Delta$ has sides $i$ and $j$, with $i$ following $j$ in the clockwise order,}\\
-1 & \text{the same holds, but in the counter-clockwise order,} \\
0 & \text{otherwise.}
\end{cases}
\]
We define
\[
\overline{Q}_\tau:=\sum_{\Delta\in \tau}\overline{Q}^\Delta,
\]
where the sum is taken over all triangles in $\tau$. 
Let $Q_\tau$ denote the quiver associated to the matrix $\overline{Q}_\tau$.

For an arc $i$ in the triangulation $\tau$, 
we can flip the edge $i$ to get a new triangulation $f_i(\tau)$.
This operation is compatible with a mutation at vertex $i$:
\[
Q_{f_i(\tau)}=\mu_i(Q_{\tau}).
\]


\subsection{Mapping Class Group Action}\label{subsec_MCG}

For a triangulation $\tau$, we define
\[
T=T(\tau):=\mathbb{C}(y_e)_{e\in \tau_1}
,\quad 
T^\vee=T^\vee(\tau):=\mathbb{C}(x_e)_{e\in \tau_1}.
\]
For a puncture $m\in M$, take a sufficiently small circle around $m$ and let $e_1,\ldots,e_n$ be the sequence of arcs which intersect with the circle, where $e_1,\ldots,e_n$ may have multiplicity.
We define
\begin{align}
y_m:=\prod_{i=1}^n y_{e_i} 
\label{def_ym}
\end{align}
and
\[
\underline{T}=\underline{T}(\tau):=\mathbb{C}[y_e,y_e^{-1}]_{e\in \tau_1}\big/(y_m)_{m\in M}.
\]

Let us fix a mapping class $\varphi$. Then the two triangulations $\tau$
and $\varphi(\tau)$ are related by a sequence of flips, together with
appropriate changes of labels.
More formally, there exists a sequence 
\[
\mathbf{k}=(k_1,\ldots,k_l)\in (\tau_1)^l
\]
such that 
the two triangulations $\tau$ and $\varphi(\tau)$ are related by the
sequence of flips associated to $\mathbf{k}$ (see \cite[Proposition
3.8]{FST}).
Note that a flip provides a canonical bijection of the edges of the
triangulations. We can represent the composition of the bijections
by a permutation $\sigma\in \mathfrak{S}_{I}$.
We define the automorphisms 
\[
\mathrm{CT}_\varphi\colon T(\tau)=T(\varphi(\tau))\overset{\sim}{\longrightarrow} T(\tau),\quad
\mathrm{CT}^\vee_\varphi\colon T^\vee(\tau)=T^\vee(\varphi(\tau))\overset{\sim}{\longrightarrow} T^\vee(\tau)
\]
by 
\[
\mathrm{CT}_\varphi(y_e)=y_{\sigma(e)}(l),\quad 
\mathrm{CT}^\vee_\varphi(x_e)=x_{\sigma(e)}(l)
\]
Thanks to the result \cite[Theorem 3.10]{FST} and the pentagon relation of cluster transformations, $\mathrm{CT}_\varphi$ and $\mathrm{CT}^\vee_\varphi$ are independent of the choices of the sequences of flips and provides a well-defined action of the mapping class group on $T(\tau)$.


\section{Pseudo-Anosov Mapping Tori}\label{sec_pA}

Let $\tau$, $\varphi$, $\mathbf{k}$ and $\sigma$ be as in \S \ref{subsec_MCG}.
We assume that no triangles are self-folded.

Let $h=h(t)$ be the edge flipped at $t$ and $h'$ be the edge
after the flip. 
Let $a$, $b$, $c$ and  $d$ be the edges of the quadrilateral in the triangulations whose diagonals are $h$ and $h'$.
We associate a topological tetrahedron $\Delta=\Delta(t)$ whose edges are labeled by $a$, $b$, $c$, $d$, $h$ and $h'$ (see Figure \ref{fig_flip_tetra}).
\begin{figure}[htb]
  \centering
  \input{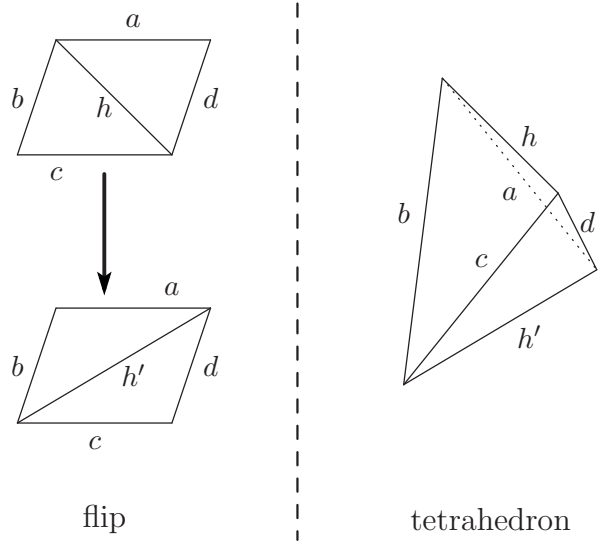}
  \caption{A flip in a 2d triangulation can be traded for a 3d tetrahedron.}
  \label{fig_flip_tetra}
\end{figure}

For any pseudo-Anosov mapping class $\varphi$, this  
provides a topological tetrahedron decompositions of the mapping
torus (\cite{agol}). A mapping class $\varphi$ is pseudo-Anosov if and only 
if the mapping torus has a hyperbolic structure.


\section{Equations for Hyperbolic Structure}\label{sec_equations}
\subsection{Shape Parameters}\label{subsec_shape}
For an ideal tetrahedron in $\mathbb{H}^3$ with vertices $0$, $1$, $z$
and $\infty$ (Figure \ref{fig_ideal}), we associate the {\it shape parameter} $z$ with the edge connecting $0$ and $\infty$. 
\begin{figure}[htbp]
  \centering
  \input{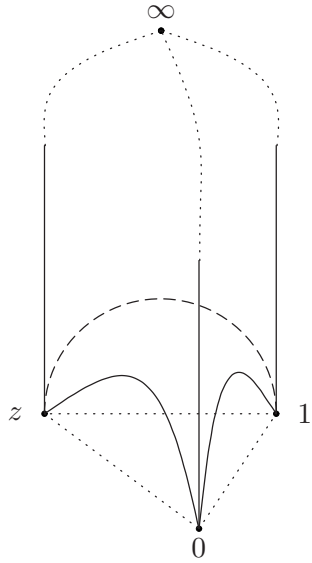}
  \caption{An ideal tetrahedron with shape parameter $z$.}
  \label{fig_ideal}
\end{figure}
For an ideal tetrahedron, a pair of mutually non-intersecting edges has a common shape parameter, and the shape parameters for the three pairs of mutually non-intersecting edges are given by (Figure \ref{fig_shape})
\begin{equation}\label{eq_parameters}
z,\quad 1-z^{-1},\quad \frac{1}{1-z} \, .
\end{equation}

\begin{figure}[htbp]
  \centering
  \scalebox{0.8}{\input{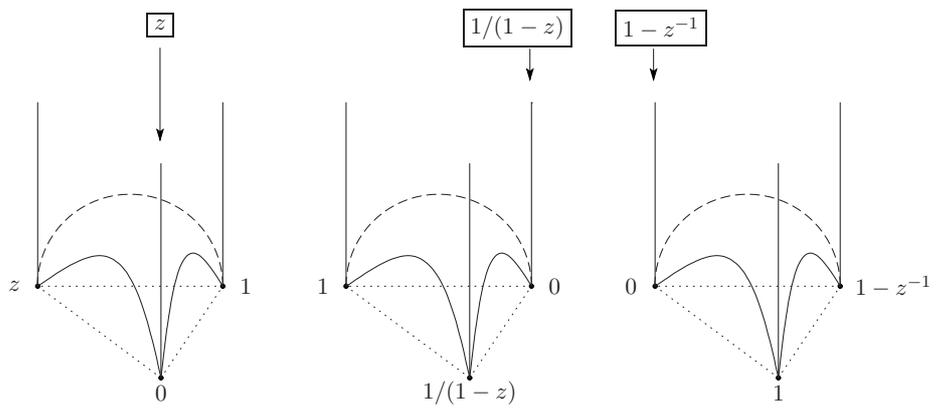}}
  \caption{The three shape parameters of an ideal tetrahedron.}
  \label{fig_shape}
\end{figure}

We take a sequence of flips and associated topological decomposition of
the mapping torus as in Section 3. 
For $t\in\mathbb{Z}$, let $\Delta(t)$ denote the $t$-th tetrahedron, where $\Delta(t)$ and $\Delta(t+l)$ are identified for any $t$.
Let $Z(t)$ denote the shape parameter of $\Delta(t)$ at the edge $h(t)$, the edge flipped at time $t$.
Note that the sequence $(Z(t))$ satisfies shape parameter periodicity 
\begin{equation}\label{eq_periodicity}
Z(t+l)=Z(t).
\end{equation}

For a tetrahedron $\Delta$, let $\Delta_1$ be the set of edges of $\Delta$.
We define
\[
\overline{E}:=\coprod_{t\in\mathbb{Z}}\Delta(t)_1.
\]
Let $\widetilde{E}$ denote the set of all edges in the tetrahedron decomposition of $\Sigma\times \mathbb{R}$ and $\pi\colon \overline{E} \to \widetilde{E}$ be the canonical surjection.

Given parameters $(Z(t))_{t\in \mathbb{Z}}$,  
we can define associated parameter $Z_e=Z_e(t)$ for any $t\in
\mathbb{Z}$ and $e\in \Delta(t)_1$ as the shape parameter of $\Delta(t)$
on the edge $e$, which is determined as in \eqref{eq_parameters}.

\subsection{Edge-gluing Conditions}\label{subsec_hyperbolic}

Suppose that the shape parameters 
$(Z(t))_{t\in \mathbb{Z}}$ 
gives an ideal tetrahedron decomposition\footnote{Here we do not require completeness. See \S \ref{subsec_complete} for complete hyperbolic structures.}.
This holds if and only if the following three conditions are satisfied.

First, we need the shape parameter periodicity condition as already discussed in
\eqref{eq_periodicity}. Second, we need 
\[
 \mathrm{Im}\, Z(t)>0 \quad  \textrm{for any} \,\, t \quad \textrm{(positivity condition)},
\]
so that the tetrahedron is positively oriented.
Third, for each edge $g\in \widetilde{E}$,
the product of all the shape parameters associated to the elements in $\pi^{-1}(g)$
must be $1$ (\cite{ThurstonLecture}, see Figure \ref{fig_structure}):
\[
\prod_{\bar{g}\in\pi^{-1}(g)}Z_{\bar{g}}=1 \quad \textrm{(edge-gluing equation)}.
\]
\begin{figure}[htbp]
  \centering
  \input{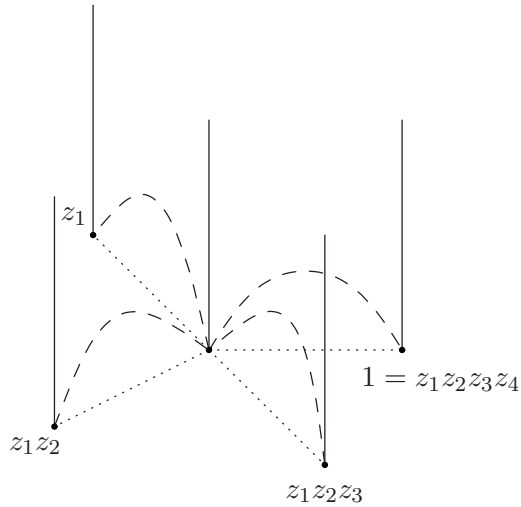}
  \caption{The edge-gluing equation around an edge.}
  \label{fig_structure}
\end{figure}

\subsection{\texorpdfstring{$y$-variables and Gluing Conditions}{y-variables and Gluing Conditions}}\label{subsec_y_structure}

\begin{prop}\label{prop_main}
Let $e(t)\in \tau(t)_1$ be the edge which we flip at $t$ and $e'(t+1)\in \tau(t+1)_1$ be the edge
 which appears after the flip. 
The edge-gluing equation is satisfied for the shape parameters
\begin{equation}\label{eq_A}
Z(t):=-y_{e(t)}(t)\left(=-y_{e'(t+1)}(t+1)^{-1}\right).
\end{equation}
\end{prop}
\begin{proof}
Let $g\in \widetilde{E}$ be an edge which appears at the $t_1$-th flip at $\bar{g}'$ and disappear at the $t_2$-th flip $\bar{g}''$ (Figure \ref{fig_edge}).
\begin{figure}[htbp]
  \centering
  \scalebox{0.85}{\includegraphics{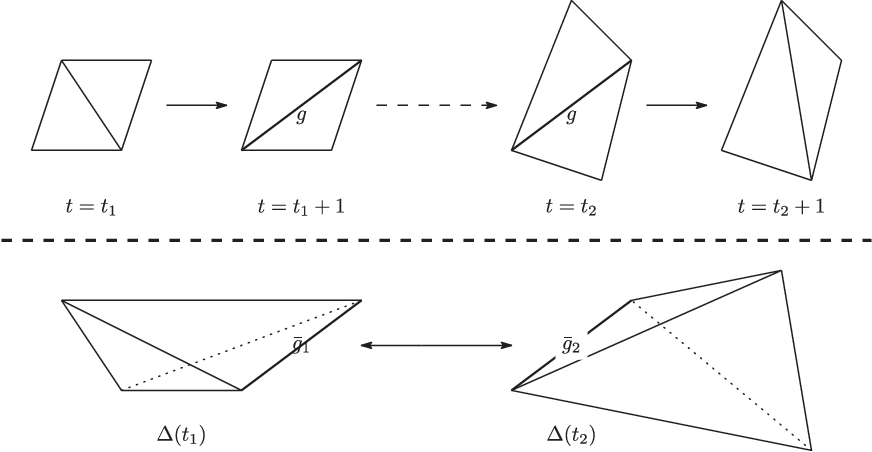}}
  \caption{An edge $g$ in a tetrahedron decomposition appears at time
 $t_1$ and disappears at time $t_2$.}
  \label{fig_edge}
\end{figure}
Let $\bar{g}_1$ (resp. $\bar{g}_2$) be the unique element in $\Delta(t_1)_1\cap \pi^{-1}(g)$ (resp. in $\Delta(t_2)_1\cap \pi^{-1}(g)$).
The gluing equation associated with $g$ is 
\begin{align*}
1&=
\prod^{t_2}_{t=t_1}\prod_{\bar{g}\in \Delta(t)_1\cap \pi^{-1}(g)}Z_{\bar{g}}\\
&=
Z(t_1)\times\left(\prod^{t_2-1}_{t=t_1+1}\prod_{\bar{g}\in \Delta(t)_1\cap \pi^{-1}(g)}Z_{\bar{g}}\right)\times Z(t_2).
\end{align*}
For this equation, we will show 
\begin{equation}\label{eq_key_lemma}
Z(t_1)\times\left(\prod^{T}_{t=t_1+1}\prod_{\bar{g}\in \Delta(t)_1\cap \pi^{-1}(g)}Z_{\bar{g}}\right)=
-y_{g_0}(T+1)^{-1},
\end{equation}
where ${g_0}$ is the edge corresponding to $g$ which $\tau(t)$
 ($t=t_1+1,\ldots,t_2$) have in common.
We show the equation above by induction with respect to $T$. 
The claim for $T=t_1$ trivially follows from the definition \eqref{eq_A}.
Let us assume the above statement for $T\to T-1$. To show the statement
 for $T$, we need to show 
\begin{align*}
\prod_{\bar{g}\in \Delta(t)_1\cap \pi^{-1}(g)}Z_{\bar{g}}&=
y_{g_0}(t)/y_{g_0}(t+1). 
\end{align*}
We will show this by classifying the positional relation of $g_0$ and $e(t)$. 
\begin{itemize}
\item $g_0$ and $e(t)$ have no triangle in common: both side of the equation above is $1$. 
\item $g_0$ and $e(t)$ have a single triangle in common :
\begin{itemize}
\item ${\overline{Q_{t}}}(e(t),g_0)=1$ (see Figure \ref{fig_+}) : 
\begin{figure}[htbp]
  \centering
  \input{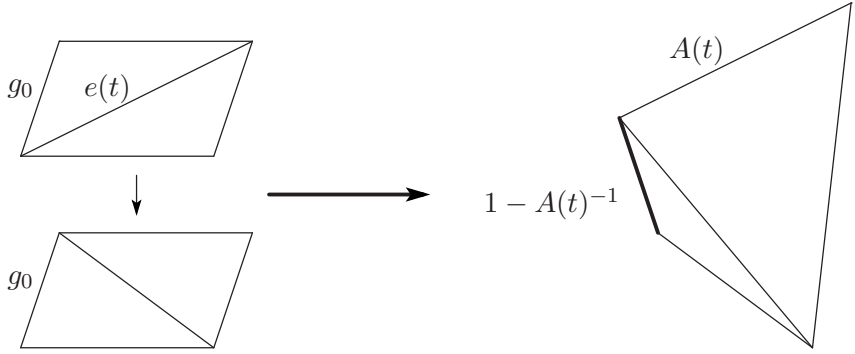}
  \caption{The case with ${\overline{Q_{t}}}(e(t),g_0)=1$.}
  \label{fig_+}
\end{figure}
\[
(\text{LHS})\overset{\text{equation} \eqref{eq_parameters}}{=} 1-Z(t)^{-1} \overset{\text{equation \eqref{eq_y}}}{=} (\text{RHS}),
\] 
\item ${\overline{Q_{t}}}(e(t),g_0)=-1$ :
\[
(\text{LHS})\overset{\text{equation} \eqref{eq_parameters}}{=} (1-Z(t))^{-1} \overset{\text{equation \eqref{eq_y}}}{=} (\text{RHS}),
\] 
\end{itemize}
\item $g_0$ and $e(t)$ have two triangles in common:
\begin{itemize}
\item ${\overline{Q_{t}}}(e(t),g_0)=\pm 2$ (see Figure \ref{fig_++}) :
\begin{figure}[htbp]
  \centering
  \scalebox{0.9}{\input{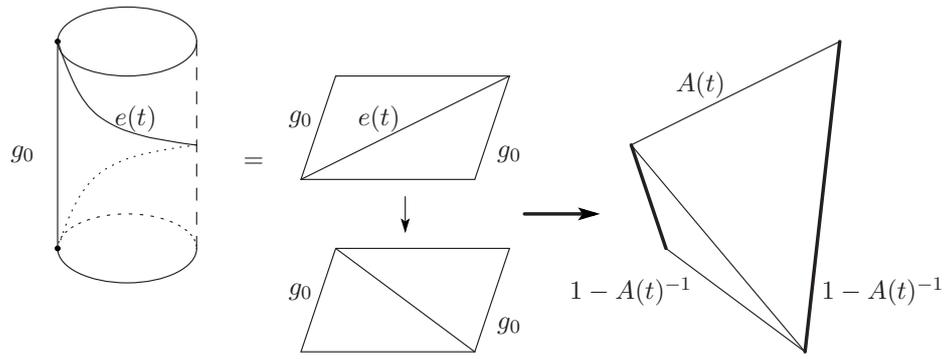}}
  \caption{The case with ${\overline{Q_{t}}}(e(t),g_0)=2$.}
  \label{fig_++}
\end{figure}
\[
(\text{LHS})\overset{\text{equation} \eqref{eq_parameters}}{=} (1-Z(t)^{\mp})^{\pm 2} \overset{\text{equation \eqref{eq_y}}}{=} (\text{RHS}),
\]
\item  
${\overline{Q_{t}}}(e(t),g_0)=0$ : this can not happen because we prohibit self-folded edges in this paper (see Figure \ref{fig_+++}).
\begin{figure}[htbp]
  \centering
  \input{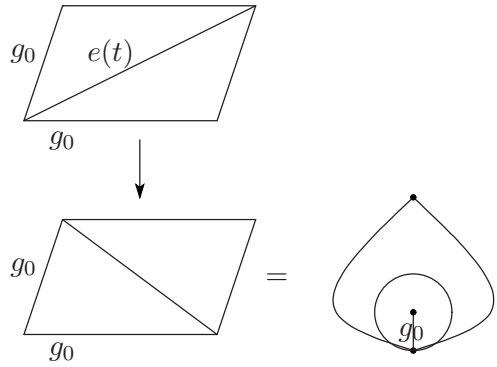}
  \caption{The case with ${\overline{Q_{t}}}(e(t),g_0)=0$.}
  \label{fig_+++}
\end{figure}
\end{itemize}
\end{itemize}

\end{proof}

\subsection{Complete Hyperbolic Structures}\label{subsec_complete}
Patching ideal tetrahedra with corners removed, 
we get a hyperbolic $3$-manifold with boundaries, each of which is isomorphic to a torus. 
Note that such a boundary torus has two directions: the direction of ``time'' parameter $t$ ({\it time direction}) and the direction of the original surface ({\it surface direction})\footnote{We avoid to use the terms ``longitude'' and ``meridian'' to avoid a confusion.}.

The intersection of a removed corner and a boundary torus gives a triangle on the torus with a shape parameter for each angle.

\begin{ex}\label{5point}
We take a five-punctured sphere. 
Let $A$, $B$, $C$, $D$, $O$ the punctures and $\sigma_1$ (resp. $\sigma_2$ or $\sigma_3$) be the Dehn half-twist along a circle containing $A$ and $B$ (resp. $B$ and $C$, or $C$ and $D$) in the anti-clock direction (see Figure \ref{fig_half}). Note that $\sigma_1$, $\sigma_2$ and $\sigma_3$ generate the braid group $B_4$. 
\begin{figure}[htbp]
  \centering
  \includegraphics{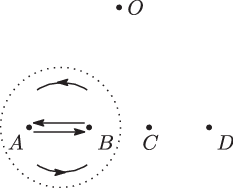}
  \caption{A Dehn half-twist $\sigma_1$ along a circle containing $A$
 and $B$.}
\label{fig_half}
\end{figure}
We take 
\begin{itemize}
\item $\sigma_1\sigma_2\sigma_3^{-1}$ as a mapping class.
\item the triangulation as in Figure \ref{fig_five_triangutation},
\item $8$, $9$, $5$, $7$, $1$, $8$ as a sequence of edges which we
      flip\footnote{Flipping at $8$, $6$ (resp. $6$, $9$, $5$, $7$ or
      $1$, $8$) corresponds to the half-twist $\sigma_1$
      (resp. $\sigma_2$ or $\sigma_3^{-1}$). Canceling the doubled $6$,
      we get the sequence above.}.
\end{itemize}
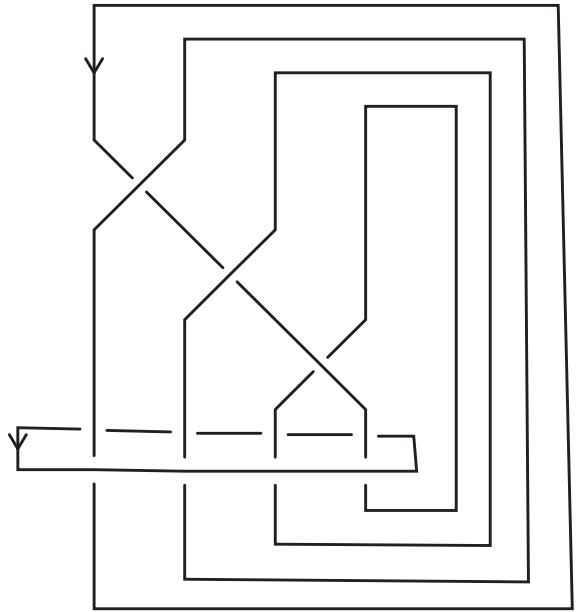
\begin{figure}[htbp]
  \centering
  \begin{center}
\unitlength 0.001in
\begin{picture}(4724,3150)(0,-3150)
    \special{pn 16}
    \special{pa 1278 278}
    \special{pa 1322 352}
    \special{pa 1366 278}
    \special{fp}
    \special{pa 1322 2498}
    \special{pa 1322 3150}
    \special{pa 3798 3150}
    \special{pa 3725 0}
    \special{pa 1322 0}
    \special{pa 1322 703}
    \special{pa 1520 901}
    \special{fp}
    \special{pa 1593 974}
    \special{pa 1989 1370}
    \special{fp}
    \special{pa 2062 1443}
    \special{pa 2494 1875}
    \special{pa 2728 2110}
    \special{pa 2728 2241}
    \special{pa 2728 2359}
    \special{fp}
    \special{pa 2728 2505}
    \special{pa 2728 2637}
    \special{pa 3197 2637}
    \special{pa 3197 527}
    \special{pa 2728 527}
    \special{pa 2728 1641}
    \special{pa 2531 1838}
    \special{fp}
    \special{pa 2457 1912}
    \special{pa 2260 2110}
    \special{pa 2260 2234}
    \special{pa 2260 2359}
    \special{fp}
    \special{pa 2260 2505}
    \special{pa 2260 2813}
    \special{pa 3373 2820}
    \special{pa 3373 352}
    \special{pa 2260 352}
    \special{pa 2260 1172}
    \special{pa 2025 1406}
    \special{pa 1791 1641}
    \special{pa 1791 2227}
    \special{pa 1791 2359}
    \special{fp}
    \special{pa 1791 2505}
    \special{pa 1791 2996}
    \special{pa 3571 3010}
    \special{pa 3549 176}
    \special{pa 1791 176}
    \special{pa 1791 703}
    \special{pa 1556 938}
    \special{pa 1322 1172}
    \special{pa 1322 2212}
    \special{pa 1322 2351}
    \special{fp}
    \special{pn 16}
    \special{pa 883 2241}
    \special{pa 927 2315}
    \special{pa 971 2241}
    \special{fp}
    \special{pa 1249 2212}
    \special{pa 927 2205}
    \special{pa 927 2424}
    \special{pa 1322 2424}
    \special{pa 1791 2432}
    \special{pa 2260 2432}
    \special{pa 2728 2432}
    \special{pa 2992 2432}
    \special{pa 2977 2249}
    \special{pa 2794 2249}
    \special{fp}
    \special{pa 2655 2241}
    \special{pa 2326 2241}
    \special{fp}
    \special{pa 2186 2234}
    \special{pa 1857 2234}
    \special{fp}
    \special{pa 1718 2227}
    \special{pa 1388 2219}
    \special{fp}
\end{picture}
\end{center}
  \caption{The link corresponding to $\sigma_1\sigma_2\sigma_3^{-1}$.}
\label{fig_link}
\end{figure}
\begin{figure}[htbp]
  \centering
  \input{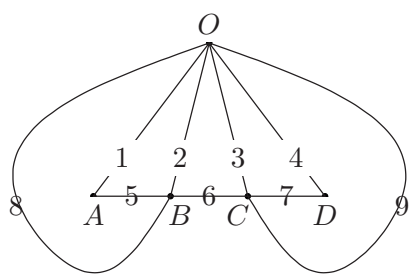}
  \caption{A triangulation of the five-punctured sphere.}
\label{fig_five_triangutation}
\end{figure}
The mapping torus is the complement of the two-component link in $S^2\times S^1$ (Figure \ref{fig_link}),
and hence we have two boundary components.
We show the triangulation of the universal cover of one of the components in Figure \ref{fig_torus}.
\begin{figure}[htbp]
  \centering
  \scalebox{0.83}{\includegraphics{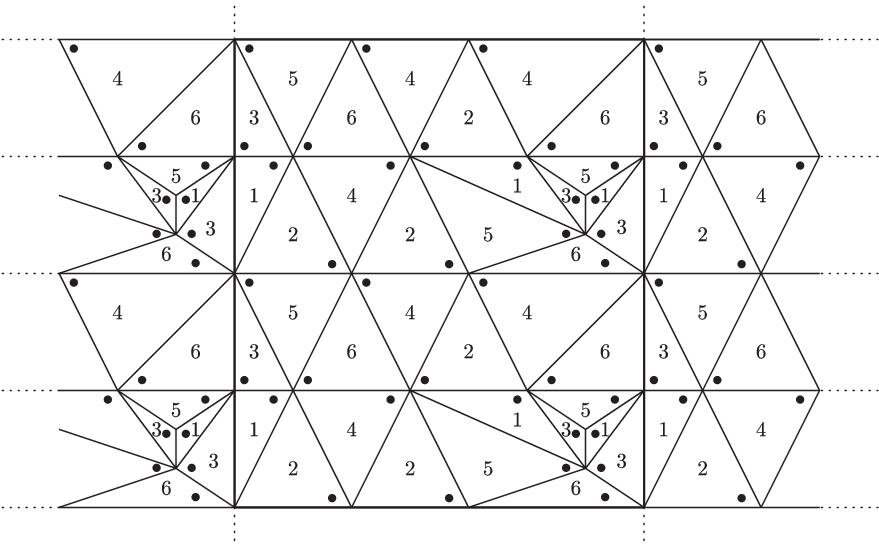}}
  \caption{The triangulation of the boundary torus. A triangle with
 number $t$ represents the $t$-th tetrahedron $\Delta(t)$, whose modulus
 $Z(t)$ corresponds to a dihedral angle represented by a black dot.}
\label{fig_torus}
\end{figure}
\end{ex}
Fix a puncture $m\in M$ of the surface and a time parameter $t_0$. 
Let $F_i$ ($i\in \mathbb{Z}/n\mathbb{Z}$) be the triangle in $\tau(t_0)$ which is adjacent to $e_{i-1}$, $m$ and $e_{i}$, 
where $(e_1,\ldots,e_n)$ is the sequence of arcs around $m$ as before.

On the boundary torus, $e_{i}$ represents a vertex and $F_i$ represents an edge connecting $e_{i-1}$ and $e_{i}$.
The union of $F_i$'s provides a (piecewise linear) closed curve 
on the boundary torus\footnote{As a cycle, this represents the homology generator in the surface direction.}. 
We call this a {\it vertical line} (see Figure \ref{fig_h_lines}).
\begin{figure}[htbp]
  \centering
  \includegraphics[scale=1]{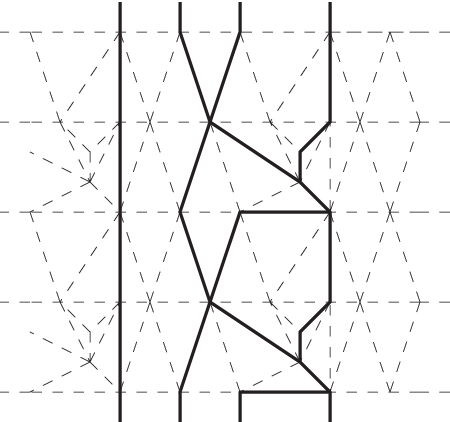}
  \caption{Vertical lines, drawn in the triangulation of Figure \ref{fig_torus}.}
\label{fig_h_lines}
\end{figure}

The {\it holonomy} of along a cycle in the surface direction is given as follows.
A vertical line divides the boundary torus into two parts.
We fix one of them.
For a vertex $e_i$ on the vertical line, we take all angles in the universal cover which have $e_i$ as the vertex and which are on the given side of the vertical line.
We denote by $H_i$ the product of the shape parameters associated to these angles.
Then we have
\begin{align}
(\text{the holonomy in the surface direction})=\prod_i(-H_i).
\label{holonomy}
\end{align}
A hyperbolic structure given by a sequence of shape parameters is complete if and only if the following condition holds: 
\begin{quote}
the holonomy along the surface direction of each boundary is trivial \quad ({\it cusp condition}, \cite{NZ_volume}).
\end{quote}

In Figure \ref{fig_edge}, we study the set of all tetrahedra which are adjacent to an edge. 
In this setting, the vertical line divides the set of these tetrahedra into two groups: tetrahedra which appear before/after $t=t_0$. Hence we have
\[
H_i=
Z(t_1)\times\left(\prod^{t_0-1}_{t=t_1+1}\prod_{\bar{e}_i\in \Delta(t)_1\cap \pi^{-1}(e_i)}Z_{\bar{e_i}}\right).
\]
By \eqref{eq_key_lemma}, the right hand side equals to $-y_{e_i}(t_0)$.
Therefore the holonomy \eqref{holonomy}
is equivalent with $y_m(t_0)=\prod_{i=1}^n y_{e_i}(t_0)$ (recall \eqref{def_ym}).
We can show this product is independent on the choice of $t_0$,
either by induction or by using the edge-gluing conditions (vertical lines at different choices of $t_0$ are homologous in
the triangulation of the boundary torus).

In summary, we get the following description of the holonomy:
\begin{prop}
For a sequence of shape parameters determined by the result of
 Proposition \ref{prop_main}, the holonomy around a puncture $m$ in the
 surface direction is equal to $y_m$.
\end{prop}

\subsection{Main Theorem}\label{subsec_main}
Let us summarize our results in the form of a theorem:
\begin{thm}\label{thm_main}
Let $(y_e)_{e\in \tau_1}$ be non-zero complex numbers such that $y_m=1$ for any puncture $m\in M$.
Assume that $y_h(t)\Big|_{y_e(0)=y_e}$ is well-defined for any $h$ and
 $t$ and
that the periodicity equation is satisfied
\[
y_{\sigma(h)} = y_h(l)\Big|_{y_e(0)=y_e}.
\]
Let us define the shape parameters $Z(t)$ by
\[
Z(t):=-y_{e(t)}(t)\Big|_{y_e(0)=y_e},
\]
where $e(t)$ is the edge flipped at time $t$,
and suppose that $Z(t)\neq 0,1$ for any $t$. 
Then $(Z(t))$ satisfies the edge-gluing equations in \S
 \ref{subsec_hyperbolic} and the cusp condition in \S \ref{subsec_complete}.
\end{thm}
This theorem gives a systematic method to identify for hyperbolic
structures on mapping tori, formulated 
in the language of cluster algebras.
For a genuine hyperbolic structure we also need to verify
$\mr{Im}(Z(t))>0$,
see the examples in the next section.


\section{Examples}
In the last section, we demonstrate Theorem \ref{thm_main} in the
case of a once-punctured torus and of a five-punctured sphere. 
The examples are chosen for the sake of simplicity, and the same methods
apply to more general mapping classes of more general punctured surfaces (recall Remark \ref{rem.KT}).
We also discuss an example of the six-punctured disc, 
to show that our formulation covers the non-hyperbolic cases 
not covered in Theorem \ref{thm_main}.

\subsection{Once-Punctured Torus and $LR$}
Let us start with a once-punctured torus.
We take a sequence of two flips as in Figure \ref{fig_LR}. 
This is the mapping class studied in \cite[\S 3.1]{TY2}.
Then the shape parameter periodicity conditions are
\begin{align*}
y_1 & = y_2^{-1}\left( 1 + y_1^{-1} (1+y_2^{-1})^{2} \right)^{-2},\\
y_2 & = y_3(1+y_2)^{2}\left(1+y_1(1+y_2^{-1})^{-2}\right)^2,\\
y_3 & = y_1^{-1}(1+y_2^{-1})^{2},
\end{align*}
and the cusp condition is
\[
y_1y_2y_3=1.
\]
Solving these equations, we get a solution
\[
y_1=1,\quad y_2=\frac{-1-\sqrt{-3}}{2},\quad y_3=\frac{-1+\sqrt{-3}}{2}.
\]
By Theorem \ref{thm_main}, shape parameters
\begin{align*}
Z(0)&=-y_2(0)=-y_2=\frac{1+\sqrt{-3}}{2}, \\
Z(1)&=-y_1(1)=-y_1(1+y_2^{-1})^{-2}=\frac{1+\sqrt{-3}}{2}, 
\end{align*}
satisfy edge-gluing conditions. 
Moreover the imaginary parts of $Z(0)$ and $Z(1)$ are positive, and
we obtain a complete hyperbolic structure on the mapping torus.
The parameters coincide with the ones in \cite[\S 3.1]{TY2}.
\begin{figure}[htbp]
  \input{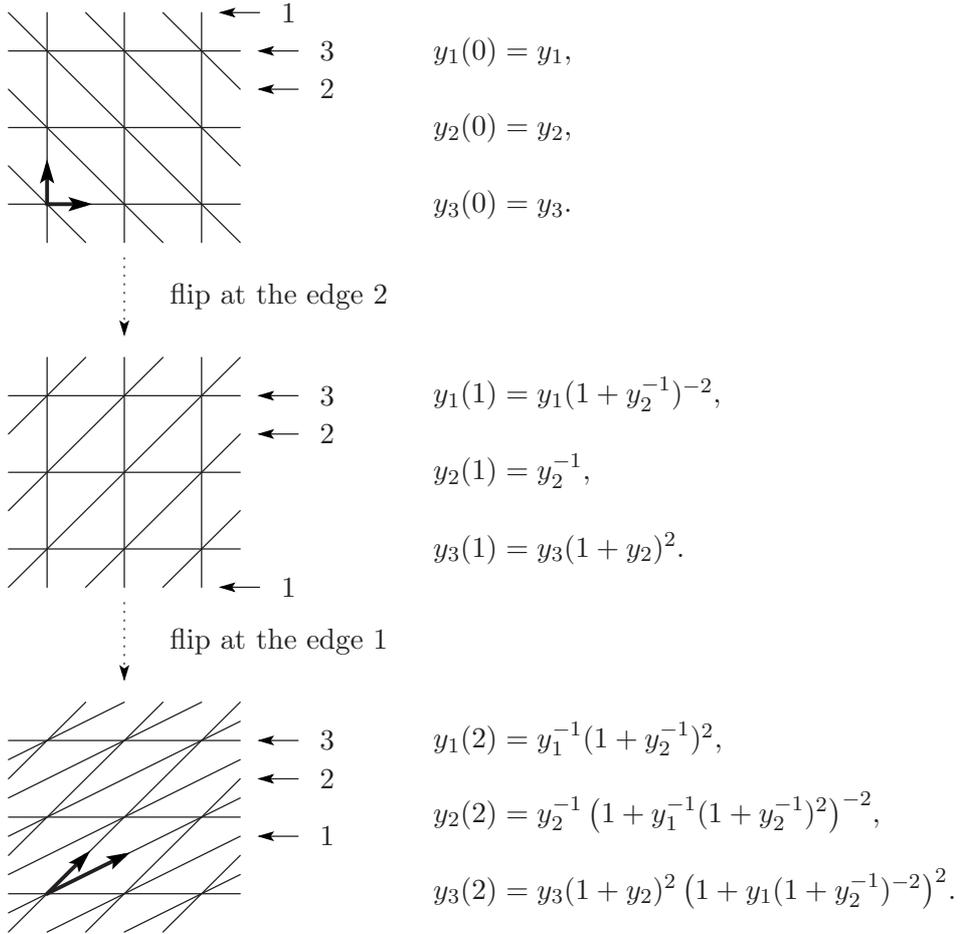}
  \caption{Example: once punctured torus and $LR$.}
  \label{fig_LR}
\end{figure}

\newpage

\subsection{\texorpdfstring{Five Punctured Sphere and $\sigma_1\sigma_2\sigma_3^{-1}$}{Five Punctured Sphere and sigma1 sigma2 sigma3(-1)}}
Let us take the example of a five-punctured sphere in Example
\ref{5point}.

The cusp conditions are
\[
y_1 y_2 y_3 y_4 y_8 y_9
=
y_1 y_5
=
y_2 y_5 y_8 y_6
=
y_3 y_6 y_9 y_7
=
y_4 y_7
=1,
\]
and the shape parameter periodicity conditions are 
{\fontsize{8}{5}\selectfont
\begin{align*}
y_6 &= 
\frac{1}{ y_1 y_5 y_8^2 y_9}
\Bigl(
 (1 + y_5 + y_5 y_8)
  ((1 + y_9 + y_7 y_9) (1 + y_9 + y_8 y_9) +\\
& \hspace{30mm}   y_5 (1 + (1 + y_7) (1 + y_8) y_9) (1 + (1 + y_8 + y_1 y_8) y_9))
\Bigr), \\
y_1 &= y_2 (1 + y_5 + y_5 y_8),\\
y_2 &= 
\frac{y_3 y_7 y_9 (1 + y_9 + y_8 y_9 +
    y_5 (1 + y_8) (1 + (1 + y_8 + y_1 y_8) y_9))}{(1 + y_9 + y_7 y_9) (1 +
     y_9 + y_8 y_9) +
  y_5 (1 + (1 + y_7) (1 + y_8) y_9) (1 + (1 + y_8 + y_1 y_8) y_9)},\\
y_9 &= 
\frac{y_4 (1 + (1 + y_7) (1 + y_8) y_9) (1 + y_9 + y_8 y_9 +
    y_5 (1 + y_8) (1 + (1 + y_8 + y_1 y_8) y_9))}{(1 + y_5 + y_5 y_8) (1 +
    y_9 + y_8 y_9)},\\
y_8 &= \frac{y_1 y_8 y_9}{
 1 + y_9 + y_8 y_9 + y_5 (1 + y_8) (1 + (1 + y_8 + y_1 y_8) y_9)},\\
y_5 &= \frac{y_5 y_6 y_8}{1 + y_5 + y_5 y_8},\\
y_4 &= \frac{(1 + y_9 + y_7 y_9) (1 + y_9 + y_8 y_9) +
  y_5 (1 + (1 + y_7) (1 + y_8) y_9) (1 + (1 + y_8 + y_1 y_8) y_9)}{
 y_7 y_8 y_9},\\
y_3 &= \frac{y_8 (1 + y_9 + y_8 y_9)}{(1 + (1 + y_7) (1 + y_8) y_9) (1 + y_9 +
    y_8 y_9 + y_5 (1 + y_8) (1 + (1 + y_8 + y_1 y_8) y_9))},\\
y_7 &= \frac{y_1 y_5 y_7 y_8 y_9}{(1 + y_9 + y_7 y_9) (1 + y_9 + y_8 y_9) +
  y_5 (1 + (1 + y_7) (1 + y_8) y_9) (1 + (1 + y_8 + y_1 y_8) y_9)}.
\end{align*}
}
Solving the shape parameter periodicity conditions with cusp conditions, we get $14$ solutions. We take one of the solutions
\begin{align*}
y_1 & = 1.781241-0.294452\times \sqrt{-1},\\
y_2 & = 1,\\
y_3 & = -0.304877 + 0.754529\times  \sqrt{-1},\\
y_4 & = 0.460355 + 1.139318\times  \sqrt{-1},\\
y_5 & = 0.546473 + 0.0903361\times  \sqrt{-1},\\
y_6 & = 1.155478 + 1.893847\times  \sqrt{-1},\\
y_7 & = 0.304877 - 0.754529\times  \sqrt{-1},\\
y_8 & = 0.304877 - 0.754529\times  \sqrt{-1},\\
y_9 & = -0.14865 - 0.664193\times  \sqrt{-1}.
\end{align*}
Following the algorithm in Theorem \ref{thm_main}, we get the following six parameters
\begin{align*}
&0.754529 \times  \sqrt{-1} -0.304877,\\
&0.754529\times  \sqrt{-1}+0.695123,\\
&0.294452\times  \sqrt{-1}-0.781241,\\
&0.754529\times  \sqrt{-1}+0.695122,\\
&0.475124\times  \sqrt{-1}+0.311704,\\
&1.139320\times  \sqrt{-1}+0.460354,
%
\end{align*}
whose imaginary parts are positive, which provide a complete hyperbolic structure.
The volume of the mapping torus computed from the parameters above is\footnote{
The hyperbolic volume of an ideal tetrahedron with modulus $z$ is given by the Bloch-Wigner function
\begin{align}
D(z)=\textrm{Im}(\textrm{Li}_2(z))+\textrm{arg}(1-z) \log|z| ,
\end{align}
where $\textrm{Li}_2(z)=-\int_0^z \frac{\log(1-t)}{t} dt$ is the 
Euler classical dilogarithm function.
When a 3-manifold is triangulated by ideal tetrahedra, the hyperbolic volume of the 3-manifold 
is the sum of the hyperbolic volumes of the tetrahedra.
}
\[
4.851170.
\]
This coincides with the value computed by SnapPea/SnapPy \cite{SnapPy}.

\subsection{Non-Hyperbolic Example}


Our formalism discussed in this paper applies to in general non-hyperbolic 3-manifolds
which are themselves not covered in Theorem \ref{thm_main}.
To illustrate this point, let us consider
a disk with six points, and we consider the $1/6$ rotation as a 
mapping class. The mapping class is realized as a sequence of $3$ 
flips as in Figure \ref{fig.disk-fig}. 

\begin{figure}[htbp]
\centering\includegraphics[scale=0.7]{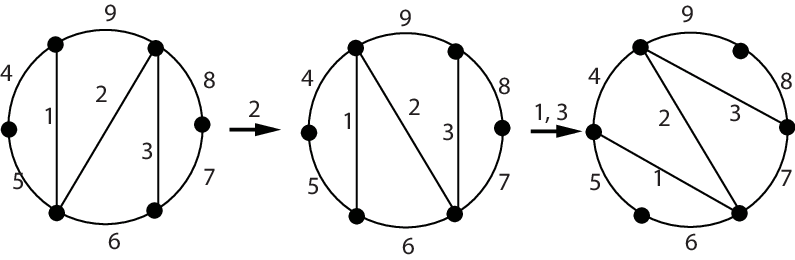}
\caption{A non-hyperbolic example, associated with the $1/6$-rotation of the six-punctured disc.}
\label{fig.disk-fig}
\end{figure}

The periodicity conditions are 
\begin{align*}
&y_1=y_1(3),\ y_2=y_2(3),\ y_3=y_3(3),\\
&y_4=y_5(3),\ y_5=y_6(3),\ y_6=y_7(3),\ y_7=y_8(3),\  
y_8=y_9(3),\ y_9=y_4(3).   
\end{align*}
Note that indices of edges in the boundary are rotated.    
The $y$-variables are given by
\begin{align*}
&y_1(3)=y_1^{-1}(1+y_2)^{-1}, \\
&y_2(3)=y_2^{-1}(1+y_1(1+y_2))(1+y_3(1+y_2)) ,\\
&y_3(3)=y_3^{-1}(1+y_2)^{-1}, \\
&y_4(3)=y_4(1+y_1^{-1}(1+y_2)^{-1})^{-1}, \\
&y_5(3)=y_5(1+y_1(1+y_2)), \\
&y_6(3)=y_6(1+y_2^{-1})^{-1}(1+y_1^{-1}(1+y_2)^{-1})^{-1}, \\
&y_7(3)=y_7(1+y_3^{-1}(1+y_2)^{-1})^{-1}, \\
&y_8(3)=y_8(1+y_3(1+y_2)), \\
&y_9(3)=y_9(1+y_2^{-1})^{-1}(1+y_3^{-1}(1+y_2)^{-1})^{-1}.
\end{align*}
A solution of the periodicity conditions is 
\begin{align*}
&y_1=\frac{1}{2}, \quad
y_2=3 , \quad
y_3=\frac{1}{2}, \quad
y_4=\sqrt{3}, \quad
y_5=\frac{1}{\sqrt{3}}, \\
&y_6=\frac{2}{\sqrt{3}}, \quad
y_7=\sqrt{3}, \quad
y_8=\frac{1}{\sqrt{3}}, \quad
y_9=\frac{2}{\sqrt{3}}.
\end{align*}
Shape parameters of three tetrahedra evaluated at the solution above are
\begin{align*}
&Z(0)=-y_2(0)=-y_2=-3
, \\
&Z(1)=-y_1(1)=-y_1(1+y_2)=-2
, \\
&Z(2)=-y_3(2)=-y_3(1+y_2)=-2
\end{align*} 
Substituting shape parameters to the Rogers dilogarithm $L(x)$, 
we have (with $Z(i)''=(1-Z(i))^{-1}$)
\begin{align*}
L(Z(0)'')+L(Z(1)'')+L(Z(2)'')&=\frac{\pi^2}{6} ,
\end{align*}
which is the complexified volume of the $3$-manifold. 
This is identified with the central charge of $\hat{\rm sl}(2)$ WZW model at 
the level $4$ (see Remark \ref{rem.TBA}). 


\bibliographystyle{amsalpha}
\bibliography{bib.bib}

\providecommand{\bysame}{\leavevmode\hbox to3em{\hrulefill}\thinspace}
\providecommand{\MR}{\relax\ifhmode\unskip\space\fi MR }
\providecommand{\MRhref}[2]{%
  \href{http://www.ams.org/mathscinet-getitem?mr=#1}{#2}
}
\providecommand{\href}[2]{#2}
\begin{thebibliography}{ADJM12}

\bibitem[ADJM12]{galaxy}
Evgeny Andriyash, Frederik Denef, Daniel~L. Jafferis, and Gregory~W. Moore,
  \emph{{Wall-crossing from supersymmetric galaxies}}, JHEP \textbf{01} (2012),
  115.

\bibitem[Ago11]{agol}
I.~Agol, \emph{Ideal triangulations of pseudo-{A}nosov mapping tori}, Topology
  and geometry in dimension three, Contemp. Math., vol. 560, Amer. Math. Soc.,
  Providence, RI, 2011, pp.~1--17.

\bibitem[AGT10]{AGT}
L.~F. Alday, D.~Gaiotto, and Y.~Tachikawa, \emph{Liouville correlation
  functions from four-dimensional gauge theories}, Lett. Math. Phys.
  \textbf{91} (2010), 167--197.

\bibitem[Ami09]{amiot}
C.~Amiot, \emph{Cluster categories for algebras of global dimension $2$ and
  quivers with potential}, {Ann. Inst. Fourier} \textbf{59} (2009), no.~6,
  2525--2590.

\bibitem[BFZ05]{fomin-zelevinsky3}
A.~Bernstein, S.~Fomin, and A.~Zelevinsky, \emph{Cluster algebras {III}: {Upper
  bounds and double Bruhat cells}}, Duke Math. J. \textbf{126} (2005), 1--52.

\bibitem[CDW]{SnapPy}
M.~Culler, N.~M. Dunfield, and J.~R. Weeks, \emph{Snappy, a computer program
  for studying the geometry and topology of 3-manifolds},
  \url{http://snappy.computop.org}.

\bibitem[CNV]{CNV}
S.~Cecotti, A.~Neitzke, and C.~Vafa, \emph{{$R$}-twisting and $4d/2d$
  correspondences}, arXiv:1006.3435.

\bibitem[DS94]{Dupont}
J.~L. Dupont and C.-H. Sah, \emph{Dilogarithm identities in conformal field
  theory and group homology}, Comm. Math. Phys. \textbf{161} (1994), no.~2,
  265--282.

\bibitem[EF12]{Eager}
Richard Eager and Sebastian Franco, \emph{{Colored BPS Pyramid Partition
  Functions, Quivers and Cluster Transformations}}, JHEP \textbf{09} (2012),
  038.

\bibitem[FG06]{fock-goncharov-teichmuller}
V.~Fock and A.~Goncharov, \emph{Moduli spaces of local systems and higher
  {Teichm\"{u}ller} theory}, Publ. Math. IHES \textbf{103} (2006), no.~1,
  1--211.

\bibitem[FG07]{fock-goncharov-dual}
\bysame, \emph{Dual {Teichm\"{u}ller} and lamination spaces}, Handbook of
  {Teichm\"{u}ller} theory \textbf{I} (2007), no.~1, 647--684, {Eur. Math.
  Soc., Zurich}.

\bibitem[FG09]{fock-goncharov-quantum-cluster}
V.~V. Fock and A.~B. Goncharov, \emph{The quantum dilogarithm and
  representations of quantum cluster varieties}, Invent. Math. \textbf{175}
  (2009), no.~2, 223--286.

\bibitem[Fom10]{fomin_ICM}
Sergey Fomin, \emph{Total positivity and cluster algebras}, Proceedings of the
  {I}nternational {C}ongress of {M}athematicians. {V}olume {II}, Hindustan Book
  Agency, New Delhi, 2010, pp.~125--145.

\bibitem[FST08]{FST}
S.~Fomin, M.~Shapiro, and D.~Thurston, \emph{Cluster algebras and triangulated
  surfaces. part {I}: Cluster complexes}, Acta Math. \textbf{201} (2008),
  no.~1, 83--146.

\bibitem[FZ02]{fomin-zelevinsky1}
S.~Fomin and A.~Zelevinsky, \emph{Cluster algebras {I}: {Foundations}}, J.
  Amer. Math. Soc. \textbf{15} (2002), no.~2, 497--529.

\bibitem[FZ03]{FZ_Ysystem}
\bysame, \emph{{$Y$}-systems and generalized associahedra}, Ann. of Math. (2)
  \textbf{158} (2003), no.~3, 977--1018.

\bibitem[GMN13]{GMN2}
Davide Gaiotto, Gregory~W. Moore, and Andrew Neitzke, \emph{{Framed BPS
  States}}, Adv. Theor. Math. Phys. \textbf{17} (2013), no.~2, 241--397.

\bibitem[GSV03]{cluster-poisson}
M.~Gekhtman, M.~Z. Shapiro, and A.~D. Vainshtein, \emph{Cluster algebras and
  poisson geometry}, Mosc. Math. J. \textbf{3} (2003), 899--934.

\bibitem[GT96]{GT}
F.~Gliozzi and R.~Tateo, \emph{Thermodynamic {B}ethe ansatz and three-fold
  triangulations}, Internat. J. Modern Phys. A \textbf{11} (1996), no.~22,
  4051--4064.

\bibitem[Gu{\'{e}}06]{one_punctured}
F.~Gu{\'{e}}ritaud, \emph{On canonical triangulations of once-punctured torus
  bundles and two-bridge link complements}, Geom. Topol. \textbf{10} (2006),
  1239--1284.

\bibitem[Ked08]{kedem}
R.~Kedem, \emph{Q-systems as cluster algebras}, J. Phys. A \textbf{41} (2008),
  194011.

\bibitem[Kel10]{keller-survay}
B.~Keller, \emph{Cluster algebras, quiver representations and triangulated
  categories}, Triangulated categories, London Math. Soc. Lecture Note Ser.,
  vol. 375, Cambridge Univ. Press, Cambridge, 2010, arXiv:0807.1960,
  pp.~76--160.

\bibitem[Kel11]{keller-q-dilog}
B.~Keller, \emph{On cluster theory and quantum dilogarithm identities},
  Representations of algebras and related topics, EMS Ser. Congr. Rep., Eur.
  Math. Soc., Z\"urich, 2011, pp.~85--116.

\bibitem[Kel13]{Keller_periodicity}
B.~Keller, \emph{The periodicity conjecture for pairs of {D}ynkin diagrams},
  Ann. of Math. (2) \textbf{177} (2013), no.~1, 111--170.

\bibitem[Kim12]{kimura}
Y.~Kimura, \emph{Quantum unipotent subgroup and dual canonical basis}, Kyoto J.
  Math. \textbf{52} (2012), no.~2, 277--331, arXiv:1010.4242.

\bibitem[KN11]{nakanishi_kashaev}
Rinat~M. Kashaev and Tomoki Nakanishi, \emph{{Classical and Quantum Dilogarithm
  Identities}}, SIGMA \textbf{7} (2011), 102.

\bibitem[KNS11]{kns}
A.~Kuniba, T.~Nakanishi, and J.~Suzuki, \emph{{T-systems and Y-systems} in
  integrable systems}, J. Phys. A \textbf{44} (2011), 103001.

\bibitem[KS]{ks}
M.~Kontsevich and Y.~Soibelman, \emph{Stability structures, motivic
  {Donaldson-Thomas} invariants and cluster transformations}, arXiv:0811.2435.

\bibitem[KT15]{kitayama-terashima}
T.~Kitayama and Y.~Terashima, \emph{Torsion functions on moduli spaces in view
  of the cluster algebra}, Geom. Dedicata \textbf{175} (2015), 125--143.

\bibitem[KY11]{dong-keller}
B.~Keller and D.~Yang, \emph{Derived equivalences from mutations of quivers
  with potential}, Adv. Math. \textbf{226} (2011), no.~3, 2118--2168,
  arXiv:0906.0761.

\bibitem[Lec10]{leclerc_ICM}
Bernard Leclerc, \emph{Cluster algebras and representation theory}, Proceedings
  of the {I}nternational {C}ongress of {M}athematicians. {V}olume {IV},
  Hindustan Book Agency, New Delhi, 2010, pp.~2471--2488.

\bibitem[Nag]{MCG_DT}
K.~Nagao, \emph{Mapping class group, {Donaldson-Thomas} theory and
  {$S$-duality}}, \url{http://www.math.nagoya-u.ac.jp/~kentaron/MCG_DT.pdf}.

\bibitem[Nag11]{dilog_id}
\bysame, \emph{Quantum dilogarithm idetities}, RIMS Kokyuroku Bessatsu
  \textbf{B28} (2011), 165--170.

\bibitem[Nag13]{cluster-via-DT}
\bysame, \emph{{Donaldson-Thomas} theory and cluster algebras}, Duke Math. J.
  \textbf{162} (2013), no.~7, 1313--1367, arXiv:1002.4884.

\bibitem[Nak11a]{nakajima-cluster}
H.~Nakajima, \emph{Quiver varieties and cluster algebras}, {Kyoto J. Math.}
  \textbf{51} (2011), no.~1, 71--126.

\bibitem[Nak11b]{Nakanishi_id}
T.~Nakanishi, \emph{Dilogarithm identities for conformal field theories and
  cluster algebras: simply laced case}, Nagoya Math. J. \textbf{202} (2011),
  23--43.

\bibitem[Nak11c]{nakanishi-periodicity}
\bysame, \emph{Periodicities in cluster algebras and dilogarithm identities},
  Representations of algebras and related topics, EMS Ser. Congr. Rep., Eur.
  Math. Soc., Z\"urich, 2011, arXiv:1006.0632, pp.~407--443.

\bibitem[Neu92]{neumann_CS}
Walter~D. Neumann, \emph{Combinatorics of triangulations and the
  {C}hern-{S}imons invariant for hyperbolic {$3$}-manifolds}, Topology '90
  ({C}olumbus, {OH}, 1990), Ohio State Univ. Math. Res. Inst. Publ., vol.~1, de
  Gruyter, Berlin, 1992, pp.~243--271. \MR{1184415}

\bibitem[NRT93]{NahmRT}
W.~Nahm, A.~Recknagel, and M.~Terhoeven, \emph{Dilogarithm identities in
  conformal field theory}, Modern Phys. Lett. A \textbf{8} (1993), no.~19,
  1835--1847.

\bibitem[NZ85]{NZ_volume}
W.~D. Neumann and D.~Zagier, \emph{Volumes of hyperbolic $3$-manifolds},
  Topology \textbf{24} (1985), 307--332.

\bibitem[Pla11]{plamondon}
P-G. Plamondon, \emph{Cluster characters for cluster categories with
  infinite-dimensional morphism spaces}, Adv. Math. \textbf{227} (2011), no.~1,
  1--39.

\bibitem[Tes07]{teschner}
J.~Teschner, \emph{An analog of a modular functor from quantized
  {Teichm\"{u}ller} theory}, Handbook of {Teichm\"{u}lle}r theory \textbf{I}
  (2007), no.~1, 685--760, {Eur. Math. Soc., Zurich}.

\bibitem[Tes11]{teschner2}
J.~Teschner, \emph{{Quantization of the Hitchin moduli spaces, Liouville
  theory, and the geometric Langlands correspondence I}}, Adv. Theor. Math.
  Phys. \textbf{15} (2011), no.~2, 471--564.

\bibitem[Thu79]{ThurstonLecture}
W.~P. Thurston, \emph{The geometry and topology of three-manifolds}, 1978-79.

\bibitem[TY11]{TY1}
Y.~Terashima and M.~Yamazaki, \emph{{SL(2,R) Chern-Simons, Liouville, and Gauge
  Theory on Duality Walls}}, JHEP \textbf{1108} (2011), 135, arXiv:1103.5748.

\bibitem[TY13]{TY2}
\bysame, \emph{{Semiclassical Analysis of the 3d/3d Relation}}, Phys.Rev.
  \textbf{D88} (2013), no.~2, 026011, arXiv:1106.3066.

\bibitem[TY14]{NTY}
\bysame, \emph{{3d N=2 Theories from Cluster Algebras}}, PTEP \textbf{023}
  (2014), B01, arXiv:1301.5902.

\end{thebibliography}

\noindent Kentaro Nagao

Graduate School of Mathematics, Nagoya University

kentaron@math.nagoya-u.ac.jp \\

\noindent Yuji Terashima

Department of Mathematical and Computing Sciences, 

Tokyo Institute of Technology

tera@is.titech.ac.jp\\

\noindent Masahito Yamazaki

Kavli Institute for the Physics and Mathematics of the Universe, 

University of Tokyo

masahito.yamazaki@ipmu.jp

\end{document}